\documentclass[11pt,reqno,a4paper]{amsart}

\usepackage{a4wide}
\usepackage{dsfont} 	
\usepackage{amssymb}
\usepackage{mathrsfs}
\usepackage{upgreek}
\usepackage{pifont} 	
\usepackage{mathabx} 	
\usepackage{verbatim}

\newtheorem{proposition}{Proposition}[section]
  \newtheorem{theorem}[proposition]{Theorem}
  \newtheorem{lemma}[proposition]{Lemma}
  \newtheorem{corollary}[proposition]{Corollary}
\theoremstyle{definition}
  
  \newtheorem{remark}[proposition]{Remark}

\newcommand{\cst}{\ifmmode\mathrm{C}^*\else{$\mathrm{C}^*$}\fi}
\newcommand{\st}{\;\vline\;} 	
\newcommand{\tens}{\otimes} 	
\newcommand{\I}{\mathds{1}}
\newcommand{\comp}{\circ}
\newcommand{\cleq}{\preccurlyeq}
\newcommand{\id}{\mathrm{id}}
\newcommand{\is}[2]{\left\langle#1\,\vline\,#2\right\rangle}
\newcommand{\Bra}[1]{\bigl\langle#1\bigr|}
\newcommand{\Ket}[1]{\bigl|#1\bigr\rangle}
\newcommand{\qqquad}{\quad\qquad}
\newcommand{\hh}[1]{\widehat{#1}}

\newcommand{\NN}{\mathbb{N}}
\newcommand{\RR}{\mathbb{R}}
\newcommand{\CC}{\mathbb{C}}
\newcommand{\ZZ}{\mathbb{Z}}
\newcommand{\GG}{\mathbb{G}}
\newcommand{\HH}{\mathbb{H}}

\newcommand{\cH}{\mathscr{H}}

\newcommand{\cc}{\text{\rm\tiny{c}}}

\DeclareMathOperator{\B}{B}
\DeclareMathOperator{\Pol}{Pol}
\DeclareMathOperator{\C}{C}
\DeclareMathOperator{\Irr}{Irr}
\DeclareMathOperator{\Mor}{Mor}

\newcommand{\bN}{\boldsymbol{\mathsf{N}}}
\newcommand{\bh}{\boldsymbol{h}}

\newcommand{\vecRho}{\overset{{}_{\scriptscriptstyle\rightarrow}}{\uprho}}
\newcommand{\ivecRho}{\overset{{}_{\scriptscriptstyle\leftarrow}}{\uprho}}
\newcommand{\WW}{\boldsymbol{W}}
\newcommand{\tp}{\!\!\!
{\scriptstyle
\text{
\raisebox{0.8pt}{
\textcircled{\raisebox{-1.7pt}{$\top$}}
} 
} 
} 
\!\!\!}
\newcommand{\stp}{\!\!\!\!\!\:
{\scriptscriptstyle
\text{
\raisebox{0.5pt}{
\textcircled{\raisebox{-1.2pt}{$\top$}}
} 
} 
} 
\!\!\!\!\!\:}

\numberwithin{equation}{section}

\DeclareMathAlphabet{\mathpzc}{OT1}{pzc}{m}{it}

\begin{document}

\author{Jacek Krajczok}
\address{Department of Mathematical Methods in Physics, Faculty of Physics, University of Warsaw, Poland}
\email{jk347906@okwf.fuw.edu.pl} 

\author{Piotr M.~So{\l}tan}
\address{Department of Mathematical Methods in Physics, Faculty of Physics, University of Warsaw, Poland}
\email{piotr.soltan@fuw.edu.pl}

\begin{abstract}
\end{abstract}

\title[Low compact quantum groups]{Compact quantum groups with representations of bounded degree}




\begin{abstract}
We show that a compact quantum group all whose irreducible representations have dimension bounded by a fixed constant must be of Kac type, in other words, its Haar measure is a trace. The proof is based on establishing several facts concerning operators related to modular properties of the Haar measure. In particular we study spectrum of these operators and the dimension of some of their eigenspaces in relation to the quantum dimension of the corresponding irreducible representation.
\end{abstract}

\maketitle

\section{Introduction}

Let $\GG$ be a compact quantum group. It is known that all irreducible representations of $\GG$ are finite dimensional. We will say that \emph{$\GG$ has representations of bounded degree} if the dimensions (in algebraic literature called \emph{degrees}) of all irreducible representations of $\GG$ are bounded by some fixed constant. This property appeared recently in the paper \cite{DSV} in connection with property $(\mathrm{T})$ for discrete quantum groups, where $\GG$ with representations of bounded degree was termed \emph{low}.\footnote{Let us also mention that in a recent preprint \cite{BK} the main results of \cite{DSV} have been established without the assumption of bounded degree of representations.} The authors of \cite{DSV} remark that compact quantum groups with representations of bounded degree exist and provide some examples (\cite[Remark 1.6]{DSV}). In fact examples of such compact quantum groups have been plentiful in non-commutative geometry (see e.g.~\cite{HM} or \cite{MB}).

Classical groups with representations of bounded degree have been studied already in \cite{kaplansky}. It was proved by C.C.~Moore in \cite{moore} that such groups must bee \emph{virtually abelian}, i.e.~they have an abelian subgroup of finite index. The interest in establishing a quantum analog of this result lead first to a much more mundane question whether a compact quantum group with representations of bounded degree must necessarily be of Kac type (have tracial Haar measure). Quite surprisingly this question turned out to be rather difficult to settle. In this paper we show that indeed a compact quantum group with representations of bounded degree is of Kac type. Compact quantum groups of Kac type are characterized in many ways e.g.~in \cite[Proposition 1.7.9]{NT}, (see also \cite[Theorem 3.4]{DiscrCpt}). The task is carried out by exploiting a number of inequalities between various numerical invariants like the quantum dimension or dimensions of certain eigenspaces of operators naturally associated with representations of quantum groups which are not of Kac type.

All necessary definitions and basic theory of compact quantum groups can be found in the book \cite{NT}. We will also follow almost all notational conventions of that book. In particular we refer the reader to \cite[Chapter 1]{NT} for the definitions of
\begin{itemize}
\item contragredient representation $U^\cc$ (\cite[Definition 1.3.8]{NT}),
\item intertwiners $\Mor(U,V)$ and self-intertwiners $\operatorname{End}(U)$ (\cite[Section 1.3]{NT}),
\item direct sums and tensor products of representations (\cite[Section 1.3]{NT}),
\item conjugate representation $\overline{U}$ (\cite[Definition 1.4.5]{NT}).
\end{itemize}
For notions related to duality between compact and discrete quantum groups (and discrete quantum groups themselves), apart from \cite{NT} we also recommend \cite{vanDaele}, \cite[Section 3]{qLor} and \cite{DiscrCpt}.

The paper is organized as follows: in Section \ref{notation} we recall certain aspects of the theory of compact quantum groups and introduce some notation needed later on. Section \ref{Symmetry} is devoted to establishing sufficient conditions for operators $\uprho_\alpha$ (see Section \ref{notation} and \cite{NT}) have eigenvalues symmetric with respect to the map $t\mapsto\tfrac{1}{t}$. In particular we show that such a symmetry statement holds for all irreducible representations of a quantum group with representations of bounded degree. In Section \ref{comult} we give a formula for the values of the comultiplication $\Delta_{\hh{\GG}}$ of the discrete quantum group $\hh{\GG}$ dual to $\GG$ on elements of the standard basis of $\operatorname{c}_{00}(\hh{\GG})$ (cf.~\cite{DiscrCpt}) which we need in the following Section \ref{spectral} dealing with spectral projections of the operators $\uprho_\alpha$. Theorem \ref{eqThm} in that section is an important technical tool for establishing our main result. The longest section \ref{mainSect} focuses on the proof of our main theorem \ref{mainThm} and finally in the appendix we briefly mention an algebraic characterization of the property of having representations of bounded degree.

\section{Notation}\label{notation}

Let $\GG$ be a compact quantum group. For a finite dimensional unitary representation $U\in\B(\cH_U)\tens\C(\GG)$ we will use the symbol $\uprho_U$ for the unique positive invertible element of $\Mor(U,U^{\cc\cc})$ such that $\operatorname{Tr}(\,\cdot\,\uprho_U)=\operatorname{Tr}(\,\cdot\,{\uprho_U}^{-1})$ on $\operatorname{End}(U)$ (\cite[Proposition 1.4.4]{NT}). Define
\[
d_t(U)=(\operatorname{Tr}\tens{f_t})(\uprho_U),\qqquad{t}\in\RR,
\]
where $f_t$ is the Woronowicz character (\cite[Definition 1.7.1]{NT}). In particular $d_1(U)$ is the \emph{quantum dimension} of $U$ (\cite[Page 16]{NT}) and $d_0(U)=\dim{U}$. Note also that $d_t(\,\cdot\,)$ is additive with respect to direct sums and multiplicative with respect to tensor products of representations.

Let $\vecRho_U$ be the list of eigenvalues of $\uprho_U$ in descending order and let $\ivecRho_U$ be the list of eigenvalues of ${\uprho_U}^{-1}$ in descending order (with possible repetitions). We will treat $\vecRho_U$ and $\ivecRho_U$ as elements of $\RR^{\dim{U}}$. Then for $t\geq{1}$ we have
\[
d_t(U)=\bigl(\|\vecRho_U\|_t\bigr)^t,
\]
where $\|\cdot\|_{t}$ is the usual $\ell_t$ norm on $\RR^{\dim{U}}$. Note that one of the defining properties of $\uprho_U$ implies that
\begin{equation}\label{traceFalpha}
\|\vecRho_U\|_1=\|\ivecRho_U\|_1.
\end{equation}

We let $\Irr{\GG}$ denote the set of equivalence classes of irreducible representations of $\GG$. For each $\alpha\in\Irr{\GG}$ we fix a unitary representative $U^\alpha\in\alpha$ on a Hilbert space $\cH_\alpha$ of dimension $n_\alpha$ (i.e.~$n_\alpha=\dim{U^\alpha}$). We write $\uprho_\alpha$ for $\uprho_{U^\alpha}$ and we fix an orthonormal basis $\{\xi^\alpha_1,\dotsc,\xi^\alpha_{n_\alpha}\}$ of $\cH_\alpha$ in which the matrix of $\uprho_\alpha$ is diagonal with descending eigenvalues. Once the basis is fixed, the corresponding matrix units will be denoted by $e_{i,j}^\alpha$:
\[
e_{i,j}^\alpha=\Ket{\bigl.\xi^\alpha_i\bigr.}\Bra{\xi^\alpha_j},\qqquad{i,j}\in\{1,\dotsc,n_\alpha\}.
\]

Throughout the paper we will use the constant
\[
\bN_\GG=\sup\{n_\alpha\st\alpha\in\Irr{\GG}\}\in\NN\cup\{+\infty\}.
\]

Some objects defined above for representations depend in fact only on the equivalence class of a given representation. In particular this is the case for $d_t(U)$, $\vecRho_U$ and $\ivecRho_U$. It follows that, with slight abuse of notation, we can write e.g. $d_t(\alpha)$, $\vecRho_\alpha$ and $\ivecRho_\alpha$ instead of $d_t(U^\alpha)$, $\vecRho_{U^\alpha}$ and $\ivecRho_{U^\alpha}$ for $\alpha\in\Irr{\GG}$. Moreover we will use the shorthand $d_t(\alpha^{\stp{n}})$ for $d_t\bigl((U^\alpha)^{\stp{n}}\bigr)$.

Throughout the paper we will assume that $\GG$ is \emph{infinite}, i.e.~$\dim{\C(\GG)}=+\infty$ (equivalently $\Irr{\GG}$ is an infinite set).

\section{Symmetry of eigenvalues}\label{Symmetry}

In this section we would like to address the situation when the eigenvalues of $\uprho_\alpha$ are symmetric in the sense that
\begin{equation}\label{symmetry}
\vecRho_\alpha=\ivecRho_\alpha.
\end{equation}
This need not be the case. Indeed one can construct compact quantum groups with the operator $\uprho_\alpha$ ``prescribed'' (at least for the so called \emph{fundamental} or \emph{defining} representation). For example, the construction of free quantum unitary as well as free quantum orthogonal groups begins with the choice of an invertible matrix $F$ and the operator $\uprho_U$ for the fundamental representation is then proportional to $(F^*F)^\top$ (cf.~\cite[Example 1.4.2]{NT}). It follows that there exists compact quantum groups with irreducible representations for which the symmetry \eqref{symmetry} does not hold. 

Nevertheless there are cases where one can prove \eqref{symmetry}. We start with a simple one.

\begin{proposition}\label{simpleSym}
Let $\alpha\in\Irr{\GG}$. If the representations $U^\alpha$ and $\overline{U^\alpha}$ are equivalent then \eqref{symmetry} holds.
\end{proposition}

\begin{proof}
This is an immediate consequence of \cite[Proposition 1.4.7]{NT} which says that $\uprho_{\overline{U^\alpha}}$ is the transpose of ${\uprho_\alpha}^{-1}$. 
\end{proof}

Proposition \eqref{simpleSym} shows for example that for $\GG=\mathrm{SU}_q(2)$ we  have \eqref{symmetry} for all $\alpha\in\Irr{\GG}$.

Our second proposition describes another situation where \eqref{symmetry} holds. To help formulate the statement, for each $\alpha\in\Irr{\GG}$ let $P_n(\alpha)$ be the maximal dimension of an irreducible subrepresentation of $\alpha^{\stp{n}}$.

\begin{proposition}\label{orderF}
Let $\alpha\in\Irr{\GG}$ and  assume the following condition is satisfied:
\begin{equation}\label{growth}
\lim_{n\to\infty}\tfrac{P_n(\alpha)}{c^n}=0,\qqquad{c}>1.
\end{equation}
Then $\vecRho_\alpha=\ivecRho_\alpha$.
\end{proposition}

In the proof we will use standard inequalities for $\ell_p$ norms on $\RR^n$:
\[
\|x\|_{p'}\leq\|x\|_p\leq{n^{\frac{1}{p}-\frac{1}{p'}}}\|x\|_{p'},\qqquad{x}\in\RR^n,\:1\leq{p}\leq{p'}<+\infty
\]
and the following elementary lemma:

\begin{lemma}\label{an1}
Let $a_1\geq{a_2}\geq\dotsm\geq{a_n}>0$ and $b_1\geq{b_2}\geq\dotsm\geq{b_m}>0$ be such that
\[
\sum_{i=1}^na_i^t=\sum_{j=1}^mb_j^t,\qqquad{t}>1.
\]
Then $n=m$ and $a_i=b_i$ for $i\in\{1,\dotsc,n\}$.
\end{lemma}

\begin{proof}[Proof of Proposition \ref{orderF}]
Let $n\in\NN$ and $\beta_1,\dotsc,\beta_{N}\in\Irr{\GG}$ be such that
\[
\alpha^{\stp{n}}\cong\bigoplus_{i=1}^N\beta_i
\]
(we are not assuming the $\beta_i$'s are pairwise non-equivalent). For each $i$ and $t>1$ we have
\[
\begin{split}
d_t(\beta_i)&=\bigl(\|\vecRho_{\beta_i}\|_t\bigr)^t\leq\bigl(\|\vecRho_{\beta_i}\|_1\bigr)^t=\bigl(\|\ivecRho_{\beta_i}\|_1\bigr)^t\\
&\leq\bigl(n_{\beta_i}^{1-\frac{1}{t}}\|\ivecRho_{\beta_i}\|_t\bigr)^t=n_{\beta_i}^{t-1}\bigl(\|\ivecRho_{\beta_i}\|_t\bigr)^t\leq{P_n(\alpha)^{t-1}}d_{-t}(\beta_i),
\end{split}
\]
where in the third step we used \eqref{traceFalpha}.

Similarly for any $i$ and $t>1$
\[
\begin{split}
d_{-t}(\beta_i)&=\bigl(\|\ivecRho_{\beta_i}\|_t\bigr)^t\leq\bigl(\|\ivecRho_{\beta_i}\|_1\bigr)^t=
\bigl(\|\vecRho_{\beta_i}\|_1\bigr)^t\\
&\leq\bigl(n_{\beta_i}^{1-\frac{1}{t}}\|\vecRho_{\beta_i}\|_t\bigr)^t=n_{\beta_i}^{t-1}\bigl(\|\vecRho_{\beta_i}\|_t\bigr)^t\leq{P_n(\alpha)}^{t-1}d_t(\beta_i).
\end{split}
\]
Summing these inequalities over $i$ yields
\[
\begin{split}
d_t(\alpha)^n&=d_t\bigl(\alpha^{\stp{n}}\bigr)\leq{P_n(\alpha)^{t-1}}d_{-t}(\alpha)^n,\\
d_{-t}(\alpha)^n&=d_{-t}\bigl(\alpha^{\stp{n}}\bigr)\leq{P_n(\alpha)^{t-1}}d_t(\alpha)^n.
\end{split}
\]
Thus
\[
\begin{split}
1\leq&\left(\frac{P_\alpha(n)}{\bigl(d_t(\alpha)/d_{-t}(\alpha)\bigr)^{\frac{n}{t-1}}}\right)^{t-1},\\
1\leq&\left(\frac{P_\alpha(n)}{\bigl(d_{-t}(\alpha)/d_t(\alpha)\bigr)^{\frac{n}{t-1}}}\right)^{t-1}
\end{split}
\]
for all $t>1$.

Now, if $d_t(\alpha)\neq{d_{-t}(\alpha)}$ for some $t>1$, then either $\bigl(d_t(\alpha)/d_{-t}(\alpha)\bigr)^\frac{1}{t-1}$ or $\bigl(d_{-t}(\alpha)/d_t(\alpha)\bigr)^\frac{1}{t-1}$ is strictly greater then $1$. Setting $c$ to be this number we get
\[
1\leq\left(\frac{P_n(\alpha)}{c^n}\right)^{t-1}.
\]
Now taking limit $n\to\infty$ gives a contradiction in the form $1\leq{0}$. It follows that we must have $d_t(\alpha)=d_{-t}(\alpha)$ for all $t>1$ and hence $\vecRho_\alpha=\ivecRho_\alpha$ by Lemma \ref{an1}.
\end{proof}

\begin{corollary}\label{symcor}
If $\bN_\GG<+\infty$ than \eqref{symmetry} holds for all $\alpha\in\Irr{\GG}$.
\end{corollary}

\begin{proof}
Obviously we have $P_n(\alpha)\leq\bN_\GG$ for all $\alpha$ and all $n$, so if $\bN_\GG<+\infty$ the condition \eqref{growth} clearly holds for all $\alpha\in\Irr{\GG}$. 
\end{proof}

\section{Comultiplication on $\operatorname{c}_0(\hh{\GG})$}\label{comult}

The dual quantum group of $\GG$ is described in detail in \cite[Section 3]{qLor} and we refer to this paper for all the details. Let us only mention that $\operatorname{c}_0(\hh{\GG})$ is by definition the \cst-algebra $\bigoplus\limits_{\alpha\in\Irr{\GG}}M_{n_\alpha}(\CC)$, so that
\begin{equation}\label{baza0}
\bigl\{e^\alpha_{i,j}{\st\alpha\in\Irr{\GG},\:i,j\in\{1,\dotsc,n_\alpha\}	\bigr\}}
\end{equation}
spans a dense $*$-subalgebra in $\operatorname{c}_0(\hh{\GG})$ which we will denote $\operatorname{c}_{00}(\hh{\GG})$. The comultiplication $\Delta_{\hh{\GG}}$ is a morphism of \cst-algebras (in the sense of \cite[Section 0]{unbo}) from $\operatorname{c}_0(\hh{\GG})$ to $\operatorname{c}_0(\hh{\GG})\tens\operatorname{c}_0(\hh{\GG})$ defined uniquely by the requirement that
\begin{equation}\label{WW}
(\Delta_{\hh{\GG}}\tens\id)\WW=\WW_{23}\WW_{13},
\end{equation}
where
\[
\WW=\bigoplus_{\alpha\in\Irr{\GG}}U^\alpha
\]
is the \emph{universal bicharacter} describing the duality between $\GG$ and $\hh{\GG}$ (see \cite[Section 3]{qLor}). In this section we will derive an explicit formula for the value of $\Delta_{\hh{\GG}}$ on elements of the basis \eqref{baza0}. For this we need to fix certain operators related to decomposition of tensor product of irreducible representations into irreducible representations.

\subsection*{Decompositions of tensor products}

For $\beta,\gamma\in\Irr{\GG}$ the tensor product $U^\beta\tp{U^\gamma}$ is equivalent to a direct sum of $U^{\alpha_1},\dotsc,U^{\alpha_n}$ with $\alpha_1,\dotsc,\alpha_n\in\Irr{\GG}$ determined uniquely up to permutation. Given $\alpha\in\Irr{\GG}$ we let $m(\alpha,\beta\tp\gamma)$ be the multiplicity of $\alpha$ in $\beta\tp\gamma$, i.e.~the number of times $\alpha$ appears in the sequence $(\alpha_1,\dotsc,\alpha_n)$ (this can be zero). Thus we have
\begin{equation}\label{gba}
U^\beta\tp{U^\gamma}\approx
\bigoplus_{\alpha\in\Irr{\GG}}\!\!\!\!\bigoplus_{i=1}^{m(\alpha,\beta\stp\gamma)}\!\!U^\alpha.
\end{equation}
Let
\[
V(\beta,\gamma):\bigoplus\limits_{\alpha\in\Irr{\GG}}\!\!\!\!\bigoplus\limits_{i=1}^{m(\alpha,\beta\stp\gamma)}\!\!\cH_\alpha\longrightarrow\cH_\beta\tens\cH_\gamma
\]
be the unitary operator implementing equivalence \eqref{gba}. Then
\[
V(\beta,\gamma)=\sum_{\alpha\in\Irr{\GG}}\!\!\!\!\sum_{i=1}^{m(\alpha,\beta\stp\gamma)}\!\!V(\alpha,\beta\tp\gamma,i)
\]
where $V(\alpha,\beta\tp\gamma,i):\cH_\alpha\to\cH_\beta\tens\cH_\gamma$ are isometries with orthogonal ranges.

The equivalence \eqref{gba} reads
\begin{equation}\label{UU}
\begin{split}
U^\beta\tp{U^\gamma}&=\bigl(V(\beta,\gamma)\tens\I\bigr)
\biggl(\bigoplus_{\alpha\in\Irr{\GG}}\!\!\!\!\bigoplus_{i=1}^{m(\alpha,\beta\stp\gamma)}\!\!U^\alpha\biggr)
\bigl(V(\beta,\gamma)^*\tens\I\bigr)\\
&=\sum_{\alpha\in\Irr{\GG}}\!\!\!\!\sum_{i=1}^{m(\alpha,\beta\stp\gamma)}\!\!
\bigl(V(\alpha,\beta\tp\gamma,i)\tens\I\bigr)U^\alpha\bigl(V(\alpha,\beta\tp\gamma,i)^*\tens\I\bigr).
\end{split}
\end{equation}

Let us write out the matrix elements of $V(\alpha,\beta\tp\gamma,i)$ in the fixed bases of the respective Hilbert spaces:
\[
V(\alpha,\beta\tp\gamma,i)\xi^\alpha_a=\sum_{b,c}V(\alpha,\beta\tp\gamma,i)^{b,c}_a\xi^\beta_b\tens\xi^\gamma_c,\qqquad{a}\in
\{1,\dotsc,n_\alpha\}
\]
(the range of indices $b$ and $c$ are self-explanatory) or in other words
\[
V(\alpha,\beta\tp\gamma,i)=\sum_{a,b,c}V(\alpha,\beta\tp\gamma,i)^{b,c}_a\Bigl(\Ket{\xi^\beta_b}\tens\Ket{\xi^\gamma_c}\Bigr)
\Bra{\xi^\alpha_a}.
\]
The coefficients $V(\alpha,\beta\tp\gamma,i)^{b,c}_a$ help express the product of matrix elements of two irreducible representations from the collection $(U^\alpha)_{\alpha\in\Irr{\GG}}$ as a linear combination of these matrix elements:

\begin{lemma}\label{ubug}
For any $\beta,\gamma\in\Irr{\GG}$ and any $b,b'\in\{1,\dotsc,n_\beta\}$, $c,c'\in\{1,\dotsc,n_\gamma\}$ we have
\begin{equation}\label{uu}
u^\beta_{b,b'}u^\gamma_{c,c'}=
\sum_{\alpha\in\Irr{\GG}}\!\!\!\!\sum_{i=1}^{m(\alpha,\beta\stp\gamma)}\!\!\sum_{a,a'}
V(\alpha,\beta\tp\gamma,i)^{b,c}_a\,u^\alpha_{a,a'}\,\overline{V(\alpha,\beta\tp\gamma,i)^{b',c'}_{a'}}.
\end{equation}
\end{lemma}

\begin{proof}
Let us expand the right hand side of formula \eqref{UU}:
{\allowdisplaybreaks
\begin{align*}
U^\beta\tp{U^\gamma}
&=\sum_{\alpha\in\Irr{\GG}}\!\!\!\!\sum_{i=1}^{m(\alpha,\beta\stp\gamma)}\!\!
\bigl(V(\alpha,\beta\tp\gamma,i)\tens\I\bigr)U^\alpha\bigl(V(\alpha\beta\tp\gamma,i)^*\tens\I\bigr)\\
&=\sum_{\alpha\in\Irr{\GG}}\!\!\!\!\sum_{i=1}^{m(\alpha,\beta\stp\gamma)}\!\!
\sum_{a,b,c}
\Bigl(V(\alpha,\beta\tp\gamma,i)^{b,c}_a\bigl(\Ket{\xi^\beta_b}\tens\Ket{\xi^\gamma_c}\bigr)\Bra{\xi^\alpha_a}\tens\I\Bigr)\\
&\qquad\biggl(\sum_{k,l}\Ket{\xi^\alpha_k}\Bra{\xi^\alpha_l}\tens{u^\alpha_{k,l}}\biggr)
\sum_{a',b',c'}
\Bigl(\,\overline{V(\alpha,\beta\tp\gamma,i)^{b',c'}_{a'}}\,\Ket{\xi^\alpha_{a'}}\bigl(\Bra{\xi^\beta_{b'}}\tens\Bra{\xi^\gamma_{c'}}\bigr)\tens\I\Bigr)\\
&=\sum_{\alpha\in\Irr{\GG}}\!\!\!\!\sum_{i=1}^{m(\alpha,\beta\stp\gamma)}\\
&\qquad\sum_{a,a',b,b',c,c'}\!\!\!\!
V(\alpha,\beta\tp\gamma,i)^{b,c}_a
\Bigl(\Ket{\xi^\beta_b}\Bra{\xi^\beta_{b'}}\tens\Ket{\xi^\gamma_c}\Bra{\xi^\gamma_{c'}}\Bigr)
\overline{V(\alpha,\beta\tp\gamma,i)^{b',c'}_{a'}}\tens{u^\alpha_{a,a'}}\\
&=\sum_{b,b',c,c'}\Bigl(\Ket{\xi^\beta_b}\Bra{\xi^\beta_{b'}}\tens\Ket{\xi^\gamma_c}\Bra{\xi^\gamma_{c'}}\Bigr)\\
&\qquad\tens\biggl(
\sum_{\alpha\in\Irr{\GG}}\!\!\!\!\sum_{i=1}^{m(\alpha,\beta\stp\gamma)}\!\!\sum_{a,a'}
V(\alpha,\beta\tp\gamma,i)^{b,c}_a\,u^\alpha_{a,a'}\,\overline{V(\alpha,\beta\tp\gamma,i)^{b',c'}_{a'}}
\biggr).
\end{align*}}

On the other hand, the left hand side of \eqref{UU} is
\[
U^\beta\tp{U^\gamma}={U^\beta}_{13}{U^\gamma}_{23}
=\sum_{b,b',c,c'}\Bigl(\Ket{\xi^\beta_b}\Bra{\xi^\beta_{b'}}\tens\Ket{\xi^\gamma_c}\Bra{\xi^\gamma_{c'}}\Bigr)\tens{u^\beta_{b,b'}}u^\gamma_{c,c'}.
\]
and \eqref{uu} follows.
\end{proof}

\subsection{Formula for $\Delta_{\hh{\GG}}$}

The proof of the next proposition owes a lot to techniques used in \cite{FLS-Fourier}.

\begin{proposition}\label{deltahat}
For any $\alpha\in\Irr{\GG}$ and any $a,a'\in\{1,\dotsc,n_\alpha\}$ we have
\[
\Delta_{\hh{\GG}}(e^\alpha_{a,a'})=
\sum_{\beta,\gamma\in\Irr{\GG}}\!\!\!\!\!
\sum_{i=1}^{m(\alpha,\beta\stp\gamma)}\!\!\!
\sum_{b,b',c,c'}
V(\alpha,\beta\tp\gamma,i)^{b,c}_a\bigl(e^\gamma_{c,c'}\tens{e^\beta_{b,b'}}\bigr)\overline{V(\alpha,\beta\tp\gamma,i)^{b',c'}_{a'}}.
\]
\end{proposition}

\begin{proof}
We expand both sides of \eqref{WW}. The left hand side is
\[
(\Delta_{\hh{\GG}}\tens\id)\WW=\sum_{\alpha\in\Irr{\GG}}\sum_{a,a'}\Delta_{\hh{\GG}}(e^\alpha_{a,a'})\tens{u^\alpha_{a,a'}},
\]
while by Lemma \ref{ubug} the right hand side is
{\allowdisplaybreaks
\begin{align*}
\WW_{23}&\WW_{13}=\sum_{\beta,\gamma\in\Irr{\GG}}\sum_{b,b',c,c'}e^\gamma_{c,c'}\tens{e^\beta_{b,b'}}\tens{u^\beta_{b,b'}}u^\gamma_{c,c'}\\
&=\sum_{\beta,\gamma\in\Irr{\GG}}\sum_{b,b',c,c'}e^\gamma_{c,c'}\tens{e^\beta_{b,b'}}\\
&\qquad\tens\biggl(\sum_{\alpha\in\Irr{\GG}}\!\!\!\!\sum_{i=1}^{m(\alpha,\beta\stp\gamma)}\!\!\sum_{a,a'}
V(\alpha,\beta\tp\gamma,i)^{b,c}_a\,u^\alpha_{a,a'}\,\overline{V(\alpha,\beta\tp\gamma,i)^{b',c'}_{a'}}\biggr)\\
&=\sum_{\alpha\in\Irr{\GG}}\sum_{a,a'}\\
&\qquad\biggl(
\sum_{\beta,\gamma\in\Irr{\GG}}\!\!\!\!\sum_{i=1}^{m(\alpha,\beta\stp\gamma)}\!\!\sum_{b,b',c,c'}
V(\alpha,\beta\tp\gamma,i)^{b,c}_a
\bigl(e^\gamma_{c,c'}\tens{e^\beta_{b,b'}}\bigr)
\overline{V(\alpha,\beta\tp\gamma,i)^{b',c'}_{a'}}
\biggr)\tens{u^\alpha_{a,a'}}
\end{align*}
}and the result follows from linear independence of $\bigl\{u^\alpha_{a,a'}{\st\alpha\in\Irr{\GG},\:a,a'\in\{1,\dotsc,n_\alpha\}\bigr\}}$.
\end{proof}

\section{Spectral projections of $\uprho_U$ operators}\label{spectral}

For a finite dimensional unitary representation $U\in\B(\cH_U)\tens\C(\GG)$ of $\GG$ and a number $t>0$ let $\uprho_U(t)$ denote the spectral projection of $\uprho_U$ corresponding to the subset $\{t\}$ of $\RR_+$, i.e.~$\uprho_U(t)=\chi_{\{t\}}(\uprho_U)$. Similarly let $\cH_U(t)$ denote the range of the projection $\uprho_U(t)$. When $U=U^\alpha$ for some $\alpha\in\Irr{\GG}$ we will write $\uprho_\alpha(t)$ and $\cH_\alpha(t)$ as usual. The properties of these spectral projections are summarized in the next proposition.

\begin{proposition}\label{rhoprop}
Let $U\in\B(\cH_U)\tens\C(\GG)$ and $V\in\B(\cH_V)\tens\C(\GG)$ be finite dimensional unitary representations of $\GG$. Then for any $t>0$ we have
\begin{enumerate}
\item\label{rhoprop1} if $T\in\Mor(U,V)$ then $T\uprho_U(t)=\uprho_V(t)T$,
\item\label{rhoprop2} $\uprho_{U\oplus{V}}(t)=\uprho_U(t)\oplus\uprho_V(t)\in\B(\cH_U)\oplus\B(\cH_V)\subset\B(\cH_U\oplus\cH_V)$,
\item\label{rhoprop3} $\uprho_{U\stp{V}}(t)=\sum\limits_{t'>0}\uprho_U(t')\tens\uprho_V(t/t')\in\B(\cH_U\tens\cH_V)$,
\item\label{rhoprop4} $\uprho_{\overline{U}}(t)=\uprho_{U}(t^{-1})^{\top}$.
\end{enumerate}
\end{proposition}

\begin{proof}
As in Section \ref{notation} we let $\{f_z\}_{z\in\CC}$ be the family of Woronowicz characters of $\GG$. Applying $(\id\tens{f_n})$ to both sides of 
\[
(T\tens \I)U=V(T\tens\I)
\]
we obtain
\[
\sum_{t>0}Tt^n\uprho_U(t)=T\uprho_{U}^n=\uprho_V^nT=\sum_{t>0}t^n\uprho_V(t)T,
\]
which implies
\[
T\uprho_U(t)=\uprho_V(t)T,\qqquad{t}>0.
\]

Points \eqref{rhoprop2}, \eqref{rhoprop3} and \eqref{rhoprop4} follow from the equalities
\begin{equation}\label{cfNT}
\uprho_{U\oplus V}=\uprho_{U}\oplus \uprho_{V},\quad
\uprho_{U\tp V}=\uprho_{U}\tens\uprho_{V}\quad\text{and}\quad
\uprho_{\overline{U}}=({\uprho_{U}}^{-1})^{\top}
\end{equation}
(see \cite[Section 1.4]{NT}).
\end{proof}

\begin{proposition}\label{mprop}
For any $\alpha,\beta,\gamma\in\Irr{\GG}$ we have
$m(\alpha,\beta\tp\gamma)=m(\beta,\alpha\tp\overline{\gamma})=m(\gamma,\overline{\beta}\tp\alpha)$.
\end{proposition}

\begin{proof}
Using \cite[Theorem 2.2.6]{NT} we obtain:
\[
\begin{split}
m(\alpha,\beta\tp\gamma)
&=\dim{\Mor(\alpha,\beta\tp\gamma)}=\dim{\Mor(\alpha\tp\overline{\gamma},\beta)}
=\dim{\Mor(\beta,\alpha\tp\overline{\gamma})}=m(\beta,\alpha\tp\overline{\gamma}),\\
m(\alpha,\beta\tp\gamma)
&=\dim{\Mor(\alpha,\beta\tp\gamma)}=\dim{\Mor(\overline{\beta}\tp\alpha,\gamma)}
=\dim{\Mor(\gamma,\overline{\beta}\tp\alpha)}=m(\gamma,\overline{\beta}\tp\alpha).
\end{split}
\]
\end{proof}

\begin{theorem}\label{eqThm}
For any $\alpha,\beta\in\Irr{\GG}$ and $s,t>0$ we have
\begin{subequations}\label{eq}
\begin{align}
\sum_{\gamma\in\Irr{\GG}}\!\!\!\!\!\sum_{i=1}^{m(\alpha,\gamma\stp\beta)}\!\!\!\!d_{\gamma}V(\alpha,\gamma\tp\beta,i)^*\bigl(\uprho_\gamma(s)\tens\uprho_\beta(t)\bigr)V(\alpha,\gamma\tp\beta,i)
&=\tfrac{d_\alpha}{t}\bigl(\dim\cH_\beta(t)\bigr)\uprho_\alpha(st),\label{eq1}\\
\sum_{\gamma\in\Irr{\GG}}\!\!\!\!\!\sum_{i=1}^{m(\alpha,\beta\stp\gamma)}\!\!\!\!d_{\gamma}V(\alpha,\beta\tp\gamma,i)^*\bigl(\uprho_\beta(t)\tens\uprho_\gamma(s)\bigr)V(\alpha,\beta\tp\gamma,i)
&=d_\alpha{t}\bigl(\dim\cH_\beta(t)\bigr)\uprho_\alpha(st).\label{eq2}
\end{align}
\end{subequations}
\end{theorem}

\begin{remark}
When $\GG$ is of Kac type equations \eqref{eq1}, \eqref{eq2} reduce to
\[
\begin{split}
\delta_{s,1}\delta_{t,1}\sum_{\gamma\in\Irr{\GG}}m(\alpha,\gamma\tp\beta)n_\gamma&=\delta_{s,1}\delta_{t,1}n_\alpha{n_\beta},\\
\delta_{s,1}\delta_{t,1}\sum_{\gamma\in\Irr{\GG}}m(\alpha,\beta\tp\gamma)n_\gamma&=\delta_{s,1}\delta_{t,1}n_\alpha{n_\beta},	
\end{split}
\]
which can be seen as an obvious equality of dimensions (see Proposition \ref{mprop})
\[
\dim\biggl(\bigoplus_{\gamma\in\Irr{\GG}}m(\gamma,\alpha\tp\overline{\beta})\cdot{U^\gamma}\biggr)
=\dim\bigl(\alpha\tp\beta\bigr)
=\dim\biggl(\bigoplus_{\gamma\in\Irr{\GG}}m(\gamma,\overline{\beta}\tp\alpha)\cdot{U^\gamma}\biggr).
\]
\end{remark}

\begin{proof}[Proof of Theorem \ref{eqThm}] 
Let $\bh_{\hh{\GG}}$ be the right Haar measure of $\hh{\GG}$:
\[
\bh_{\hh{\GG}}(x)=\sum_{\alpha\in\Irr{\GG}}d_\alpha\operatorname{Tr}\bigl(\uprho_\alpha\pi_\alpha(x)\bigr),\qqquad{x}\in\operatorname{c}_{00}(\hh{\GG}),
\]
where $\pi_\alpha:\operatorname{c}_{00}(\hh{\GG})\to\B(\cH_\alpha)$ is the canonical projection (cf.~\cite[Section 3]{qLor}). Note that 
\begin{equation}\label{he}
\bh_{\hh{\GG}}(e^\alpha_{a,a'})=
\sum_{\beta\in\Irr{\GG}}d_\beta\operatorname{Tr}_\beta(\uprho_\beta\,e^\alpha_{a,a'})
=d_\alpha\sum_{a''}\is{\xi^\alpha_{a''}}{\uprho_\beta\,e^\alpha_{a,a'}\,\xi^\alpha_{a''}}
=d_\alpha\is{\xi^\alpha_{a'}}{\uprho_\alpha\,\xi^\alpha_{a}}
\end{equation}
for any $\alpha\in\Irr{\GG}$ and $a,a'\in\{1,\dotsc,n_\alpha\}$.

Modular properties of $\bh_{\hh{\GG}}$ are described by operators $\{\uprho_\alpha\}_{\alpha\in\Irr{\GG}}$:
\begin{subequations}
\begin{align}
(\id\tens\bh_{\hh{\GG}})\comp\Delta_{\hh{\GG}}&=\bh_{\hh{\GG}}(\,\cdot\,)\bigoplus_{\alpha\in\Irr{\GG}}{\uprho_\alpha}^{-2},\label{modu1}\\
(\bh_{\hh{\GG}}\tens\id)\comp\Delta_{\hh{\GG}}&=\bh_{\hh{\GG}}(\,\cdot\,)\I,\label{modu2}
\end{align}
\end{subequations}
(\cite[Theorem 3.3]{qLor}).

Choose $\alpha\in\Irr{\GG}$ and $a,a'\in\{1,\dotsc,n_\alpha\}$. Using \eqref{he}, the formula for $\Delta_{\hh{\GG}}$ given in Proposition \ref{deltahat} and \eqref{modu1} we get
\begin{equation}\label{ppeq1}
\begin{split}
&\sum_{\beta,\gamma\in\Irr{\GG}}\!\!\!\!\!\!\sum_{i=1}^{m(\alpha,\beta\stp\gamma)}\!\!\!\!\sum_{b,b',c,c'}
V(\alpha,\beta\tp\gamma,i)^{b,c}_a\,e^\gamma_{c,c'}\,d_\beta\is{\xi^\beta_{b'}}{\uprho_\beta\,\xi^\beta_{b}}\overline{V(\alpha,\beta\tp\gamma,i)^{b',c'}_{a'}}\\
&\qquad=\sum_{\beta,\gamma\in\Irr{\GG}}\!\!\!\!\!\!\sum_{i=1}^{m(\alpha,\beta\stp\gamma)}\!\!\!\!\sum_{b,b',c,c'}
V(\alpha,\beta\tp\gamma,i)^{b,c}_a\,e^\gamma_{c,c'}\,\bh_{\hh{\GG}}(e^\beta_{b,b'})\overline{V(\alpha,\beta\tp\gamma,i)^{b',c'}_{a'}}\\
&\qquad=(\id\tens\bh_{\hh{\GG}})\Delta_{\hh{\GG}}(e^\alpha_{a,a'})
=\bh_{\hh{\GG}}(e^\alpha_{a,a'})\bigoplus_{\gamma\in\Irr{\GG}}{\uprho_\gamma}^{-2}\\
&\qquad=\sum_{\gamma\in\Irr{\GG}}\sum_{c,c'}d_\alpha\is{\xi^\alpha_{a'}}{\uprho_\alpha\,\xi^\alpha_{a}}\is{\xi^{\gamma}_{c}}{{\uprho_\gamma}^{-2}\,\xi^{\gamma}_{c'}}e^{\gamma}_{c,c'}
\end{split}
\end{equation}
and similarly (this time using \eqref{modu2})
\begin{equation}\label{ppeq2}
\begin{split}
&\sum_{\beta,\gamma\in\Irr{\GG}}\!\!\!\!\!\!\sum_{i=1}^{m(\alpha,\beta\stp\gamma)}\!\!\!\!\sum_{b,b',c,c'}
V(\alpha,\beta\tp\gamma,i)^{b,c}_ad_\gamma\is{\xi^{\gamma}_{c'}}{\uprho_\gamma\,\xi^{\gamma}_{c}}e^\beta_{b,b'}\overline{V(\alpha,\beta\tp\gamma,i)^{b',c'}_{a'}}\\
&\qquad=\sum_{\beta,\gamma\in\Irr{\GG}}\!\!\!\!\!\!\sum_{i=1}^{m(\alpha,\beta\stp\gamma)}\!\!\!\!\sum_{b,b',c,c'}
V(\alpha,\beta\tp\gamma,i)^{b,c}_a\bh_{\hh{\GG}}(e^\gamma_{c,c'})e^\beta_{b,b'}\overline{V(\alpha,\beta\tp\gamma,i)^{b',c'}_{a'}}\\
&\qquad=(\bh_{\hh{\GG}}\tens\id)\Delta_{\hh{\GG}}(e^\alpha_{a,a'})=\bh_{\hh{\GG}}(e^\alpha_{a,a'})\I\\
&\qquad=\sum_{\beta\in\Irr{\GG}}\sum_{b,b'}d_\alpha\is{\xi^\alpha_{a'}}{\uprho_\alpha\,\xi^\alpha_{a}}\delta_{b,b'}e^\beta_{b,b'}.
\end{split}
\end{equation}

If we equate appropriate coefficients in equations \eqref{ppeq1}, \eqref{ppeq2} we get
\begin{equation}\label{ppeq3}
\begin{split}
\sum_{\beta\in\Irr{\GG}}\!\!\!\!\sum_{i=1}^{m(\alpha,\beta\stp\gamma)}\!\!\sum_{b,b'}
&V(\alpha,\beta\tp\gamma,i)^{b,c}_ad_\beta\is{\xi^\beta_{b'}}{\uprho_\beta\,\xi^\beta_{b}}\overline{V(\alpha,\beta\tp\gamma,i)^{b',c'}_{a'}}\\
&=d_\alpha\is{\xi^\alpha_{a'}}{\uprho_\alpha\,\xi^\alpha_{a}}\is{\xi^\gamma_c}{{\uprho_\gamma}^{-2}\,\xi^\gamma_{c'}}
\end{split}
\end{equation}
and
\begin{equation}\label{ppeq4}
\begin{split}
\sum_{\gamma\in\Irr{\GG}}\!\!\!\!\sum_{i=1}^{m(\alpha,\beta\stp\gamma)}\!\!\sum_{c,c'}
&V(\alpha,\beta\tp\gamma,i)^{b,c}_ad_\gamma\is{\xi^\gamma_{c'}}{\uprho_\gamma\,\xi^\gamma_c}\overline{V(\alpha,\beta\tp\gamma,i)^{b',c'}_{a'}}\\
&=d_\alpha\is{\xi^\alpha_{a'}}{\uprho_\alpha\,\xi^\alpha_a}\delta_{b,b'}.
\end{split}
\end{equation}

Now let us fix $t>0$, multiply both sides of \eqref{ppeq3} by $\is{\uprho_\gamma(t)\,\xi^\gamma_{c'}}{\xi^\gamma_c}$ and sum over $c,c'\in\{1,\dotsc,n_\gamma\}$.
\begin{equation}\label{ppeq5}
\begin{split}
\sum_{\beta\in\Irr{\GG}}\!\!\!\!\sum_{i=1}^{m(\alpha,\beta\stp\gamma)}\!\!\!\!\sum_{b,b',c,c'}
&V(\alpha,\beta\tp\gamma,i)^{b,c}_ad_\beta\is{\uprho_\gamma(t)\,\xi^\gamma_{c'}}{\xi^\gamma_c}\is{\xi^\beta_{b'}}{\uprho_\beta\,\xi^\beta_b}\overline{V(\alpha,\beta\tp\gamma,i)^{b',c'}_{a'}}\\
&=\sum_{c,c'}d_\alpha\is{\xi^\alpha_{a'}}{\uprho_\alpha\,\xi^\alpha_a}\is{\uprho_\gamma(t)\,\xi^\gamma_{c'}}{\xi^\gamma_c}\is{\xi^\gamma_c}{{\uprho_\gamma}^{-2}\,\xi^\gamma_{c'}}
\end{split}
\end{equation}
and similarly multiplying both sides of \eqref{ppeq4} by $\is{\uprho_\beta(t)\,\xi^\beta_{b'}}{\xi^\gamma_b}$ and summing over $b,b'\in\{1,\dotsc,n_\beta\}$ we get
\begin{equation}\label{ppeq6}
\begin{split}
\sum_{\gamma\in\Irr{\GG}}\!\!\!\!\sum_{i=1}^{m(\alpha,\beta\stp\gamma)}\!\!\!\!\sum_{b,b',c,c'}
&V(\alpha,\beta\tp\gamma,i)^{b,c}_ad_\gamma
\is{\uprho_\beta(t)\,\xi^\beta_{b'}}{\xi^\beta_b}\is{\xi^\gamma_{c'}}{\uprho_\gamma\,\xi^\gamma_c}\overline{V(\alpha,\beta\tp\gamma,i)^{b',c'}_{a'}}\\
&=\sum_{b,b'}d_\alpha\is{\xi^\alpha_{a'}}{\uprho_\alpha\,\xi^\alpha_a}\is{\uprho_\beta(t)\,\xi^\beta_{b'}}{\xi^\beta_b}\delta_{b,b'}.
\end{split}
\end{equation}

Let us rewrite the left hand side of \eqref{ppeq5} in the following way:
{\allowdisplaybreaks
\begin{align*}
&\sum_{\beta\in\Irr{\GG}}\!\!\!\!\sum_{i=1}^{m(\alpha,\beta\stp\gamma)}\!\!\!\!\sum_{b,b',c,c'}
V(\alpha,\beta\tp\gamma,i)^{b,c}_ad_\beta\is{\uprho_\gamma(t)\,\xi^\gamma_{c'}}{\xi^\gamma_c}\is{\xi^\beta_{b'}}{\uprho_\beta\,\xi^\beta_b}\overline{V(\alpha,\beta\tp\gamma,i)^{b',c'}_{a'}}\\
&\quad=\sum_{\beta\in\Irr{\GG}}\!\!\!\!\sum_{i=1}^{m(\alpha,\beta\stp\gamma)}\!\!\!\!\sum_{b,b',c,c'}
V(\alpha,\beta\tp\gamma,i)^{b,c}_ad_\beta\is{\bigl(\uprho_\beta\tens\uprho_\gamma(t)\bigr)(\xi^\beta_{b'}\tens\xi^\gamma_{c'})}{\xi^\beta_b\tens\xi^\gamma_c}\overline{V(\alpha,\beta\tp\gamma,i)^{b',c'}_{a'}}\\
&\quad=\sum_{\beta\in\Irr{\GG}}\!\!\!\!\!\sum_{i=1}^{m(\alpha,\beta\stp\gamma)}\!\!\!\!
d_\beta\is{\bigl(\uprho_\beta\tens\uprho_\gamma(t)\bigr)V(\alpha,\beta\tp\gamma,i)\xi^\alpha_{a'}}{V(\alpha,\beta\tp\gamma,i)\xi^\alpha_a}\\
&\quad=\is{\xi^\alpha_{a'}}
{\sum_{\beta\in\Irr{\GG}}\!\!\!\!\!\sum_{i=1}^{m(\alpha,\beta\stp\gamma)}\!\!\!\!d_{\beta}V(\alpha,\beta\tp\gamma,i)^*\bigl(\uprho_\beta\tens\uprho_\gamma(t)\bigr)V(\alpha,\beta\tp\gamma,i)\xi^\alpha_a}
\end{align*}
}and similarly with the left hand side of \eqref{ppeq6}:
\[
\begin{split}
&\sum_{\gamma\in\Irr{\GG}}\!\!\!\!\sum_{i=1}^{m(\alpha,\beta\stp\gamma)}\!\!\!\!\sum_{b,b',c,c'}
V(\alpha,\beta\tp\gamma,i)^{b,c}_ad_\gamma\is{\uprho_\beta(t)\,\xi^\beta_{b'}}{\xi^\gamma_b}\is{\xi^\gamma_{c'}}{\uprho_\gamma\,\xi^\gamma_c}\overline{V(\alpha,\beta\tp\gamma,i)^{b',c'}_{a'}}\\
&\quad=\is{\xi^\alpha_{a'}}
{\sum_{\gamma\in\Irr{\GG}}\!\!\!\!\!\sum_{i=1}^{m(\alpha,\beta\stp\gamma)}\!\!\!\!d_{\gamma}V(\alpha,\beta\tp\gamma,i)^*\bigl(\uprho_\beta(t)\tens\uprho_\gamma\bigr)V(\alpha,\beta\tp\gamma,i)\xi^\alpha_a}.
\end{split}
\]
Now the right hand sides of \eqref{ppeq5} and \eqref{ppeq6} are respectively
\[
\begin{split}
\sum_{c,c'}d_\alpha\is{\xi^\alpha_{a'}}{\uprho_\alpha\,\xi^\alpha_a}\is{\uprho_\gamma(t)\,\xi^\gamma_{c'}}{\xi^\gamma_c}\is{\xi^\gamma_c}{{\uprho_\gamma}^{-2}\,\xi^\gamma_{c'}}
&=d_\alpha\is{\xi^\alpha_{a'}}{\uprho_\alpha\,\xi^\alpha_a}\sum_{c'}\is{{\uprho_\gamma}^{-2}\uprho_\gamma(t)\,\xi^\gamma_{c'}}{\xi^\gamma_{c'}}\\
=d_\alpha\is{\xi^\alpha_{a'}}{\uprho_\alpha\,\xi^\alpha_a}t^{-2}\dim{\cH_\gamma(t)}
&=\is{\xi^\alpha_{a'}}{d_\alpha{t^{-2}}\bigl(\dim{\cH_\gamma(t)}\bigr)\uprho_\alpha\,\xi^\alpha_a}
\end{split}
\]
and
\[
\sum_{b,b'}d_\alpha\is{\xi^\alpha_{a'}}{\uprho_\alpha\,\xi^\alpha_a}\is{\uprho_\beta(t)\,\xi^\beta_{b'}}{\xi^\beta_b}\delta_{b,b'}=
\is{\xi^\alpha_{a'}}{d_\alpha\bigl(\dim{\cH_\beta(t)}\bigr)\uprho_\alpha\,\xi^\alpha_a}.
\]
Combining the above results and using the fact that $\{\xi^\alpha_1,\dotsc,\xi^\alpha_{n_\alpha}\}$ is a basis of $\cH_\alpha$ we arrive at
\begin{equation}\label{ppeq7}
\sum_{\beta\in\Irr{\GG}}\!\!\!\!\!\sum_{i=1}^{m(\alpha,\beta\stp\gamma)}\!\!\!\!d_{\beta}V(\alpha,\beta\tp\gamma,i)^*\bigl(\uprho_\beta\tens\uprho_\gamma(t)\bigr)V(\alpha,\beta\tp\gamma,i)
=d_\alpha{t^{-2}}\bigl(\dim{\cH_\gamma(t)}\bigr)\uprho_\alpha,
\end{equation}
and
\begin{equation}\label{ppeq8}
\sum_{\gamma\in\Irr{\GG}}\!\!\!\!\!\sum_{i=1}^{m(\alpha,\beta\stp\gamma)}\!\!\!\!d_{\gamma}V(\alpha,\beta\tp\gamma,i)^*\bigl(\uprho_\beta(t)\tens\uprho_\gamma\bigr)V(\alpha,\beta\tp\gamma,i)
=d_\alpha\bigl(\dim{\cH_\beta(t)}\bigr)\uprho_\alpha.
\end{equation}

Take arbitrary $s>0$. From Proposition \ref{rhoprop}\eqref{rhoprop3} we know that
\[
\begin{split}
\bigl(\uprho_\beta\tens\uprho_\gamma(t)\bigr)V(\alpha,\beta\tp\gamma,i)
&=s\bigl(\uprho_\beta(s)\tens\uprho_\gamma(t)\bigr)V(\alpha,\beta\tp\gamma,i),\\
\bigl(\uprho_\beta(t)\tens\uprho_\gamma)V(\alpha,\beta\tp\gamma,i)
&=s\bigl(\uprho_\beta(t)\tens\uprho_\gamma(s)\bigr)V(\alpha,\beta\tp\gamma,i)
\end{split}
\]
on $\cH_\alpha(st)$, while on $\cH_\alpha(st)^{\perp}$ we have
\[
\bigl(\uprho_\beta(s)\tens\uprho_\gamma(t)\bigr)V(\alpha,\beta\tp\gamma,i)=\bigl(\uprho_\beta(t)\tens\uprho_\gamma(s)\bigr)V(\alpha,\beta\tp\gamma,i)=\uprho_\alpha(st)=0.
\]
Therefore, we can rewrite equations \eqref{ppeq7} and \eqref{ppeq8} as
\[
\begin{split}
\sum_{\beta\in\Irr{\GG}}\!\!\!\!\!\sum_{i=1}^{m(\alpha,\beta\stp\gamma)}\!\!\!\!d_{\beta}V(\alpha,\beta\tp\gamma,i)^*\bigl(\uprho_\beta(s)\tens\uprho_\gamma(t)\bigr)V(\alpha,\beta\tp\gamma,i)
&=\tfrac{d_\alpha}{t}\bigl(\dim{\cH_\gamma(t)}\bigr)\uprho_\alpha(st),\\
\sum_{\gamma\in\Irr{\GG}}\!\!\!\!\!\sum_{i=1}^{m(\alpha,\beta\stp\gamma)}\!\!\!\!d_{\gamma}V(\alpha,\beta\tp\gamma,i)^*\bigl(\uprho_\beta(t)\tens\uprho_\gamma(s)\bigr)V(\alpha,\beta\tp\gamma,i)
&=d_\alpha{t}\bigl(\dim{\cH_\beta(t)}\bigr)\uprho_\alpha(st),
\end{split}
\]
which is \eqref{eq} (after exchanging $\beta\leftrightarrow\gamma$ in the first equation).
\end{proof}

\section{Bounded degree of representations implies Kac type}\label{mainSect}

Before proceeding with our main result (Theorem \ref{mainThm}) let us introduce the following useful notation. For a finite dimensional unitary representation $U\in\B(\cH_{U})\tens\C(\GG)$ we will write $\Gamma(U)$ for the maximal eigenvalue of $\uprho_U$ which is also equal to the operator norm of $\uprho_U$, i.e.~$\Gamma(U)=\|\uprho_{U}\|$. As usual, whenever $U=U^\alpha$ for $\alpha\in\Irr{\GG}$ we will write $\Gamma(\alpha)$ instead of $\Gamma(U^\alpha)$ (this makes perfect sense, since $\Gamma(U)$ depends only on equivalence class of $U$). The following proposition describes some properties of the map $U\mapsto\Gamma(U)$:

\begin{proposition}\label{gammaprop}
Let $U\in\B(\cH_{U})\tens\C(\GG)$ and $V\in\B(\cH_{V})\tens\C(\GG)$ be finite dimensional unitary representations of $\GG$. We have
\begin{enumerate}
\item\label{gammaprop1} $\Gamma(U\oplus V)=\max\{\Gamma(U),\Gamma(V)\}$,
\item\label{gammaprop2} $\Gamma(U\tp V)=\Gamma(U)\Gamma(V)$,
\item\label{gammaprop3} if eigenvalues of $\uprho_{U}$ are symmetric, that is $\vecRho_U=\ivecRho_U$, then $\Gamma(U)=\Gamma(\overline{U})$.
\end{enumerate}
\end{proposition}

\begin{proof}
Properties \eqref{gammaprop1}--\eqref{gammaprop3} all follow from \eqref{cfNT}.
\end{proof}

The next result will be needed in the proof of Theorem \ref{mainThm}. In what follows, for $\alpha,\beta,\gamma\in\Irr{\GG}$, we will write $\gamma\cleq\alpha\tp\beta$ if $m(\gamma,\alpha\tp\beta)\neq{0}$.

\begin{proposition}\label{dimgammaineq}
Let $\alpha,\beta,\gamma\in\Irr{\GG}$ be such that $\gamma\cleq\alpha\tp\beta$, $\Gamma(\gamma)=\Gamma(\alpha)\Gamma(\beta)$ and 
\[
\tfrac{d_\gamma}{\dim{\cH_\gamma(\Gamma(\gamma))}}
=\max\Bigl\{\tfrac{d_{\gamma'}}{\dim{\cH_{\gamma'}(\Gamma(\gamma'))}}\st\gamma'\in\Irr{\GG},\:\gamma'\cleq\alpha\tp\beta,\:\Gamma(\gamma')=\Gamma(\alpha)\Gamma(\beta)\Bigr\}.
\]
Then
\begin{equation}\label{dimgammaineq1}
1\leq\frac{d_\gamma\dim{\cH_\alpha\bigl(\Gamma(\alpha)\bigr)}}{d_\alpha\Gamma(\beta)\dim{\cH_\gamma\bigl(\Gamma(\gamma)\bigr)}}.
\end{equation}
\end{proposition}

Let us note that $\gamma$ as in Proposition \ref{dimgammaineq} always exists. Indeed, the representation $U^\alpha\tp{U^\beta}$ is equivalent to $U^{\gamma_1}\oplus\dotsm\oplus{U^{\gamma_n}}$ for some $\gamma_1,\dotsc,\gamma_n\in\Irr{\GG}$ (possibly with repetitions). By Proposition \eqref{gammaprop} we have
\[
\Gamma(\alpha)\Gamma(\beta)=\Gamma(\alpha\tp\beta)=\max\bigl\{\Gamma(\gamma_1),\dotsc,\Gamma(\gamma_n)\bigr\},
\]
so there must exist $\gamma\in\Irr{\GG}$ such that $\gamma\cleq\alpha\tp\beta$ and $\Gamma(\gamma)=\Gamma(\alpha)\Gamma(\beta)$.

\begin{proof}[Proof of Proposition \ref{dimgammaineq}]
First equality of Theorem \ref{eqThm} and the fact that $\uprho_{\overline{\beta}}=({\uprho_\beta}^{-1})^{\top}$ implies
\begin{equation}\label{najpierw}
\begin{split}
\sum_{\gamma'\in\Irr{\GG}}\!\!\!\!\!\sum_{i=1}^{m(\alpha,\gamma'\stp\overline{\beta})}\!\!\!&d_{\gamma'}V(\alpha,\gamma'\tp\overline{\beta},i)^*\bigl(\uprho_{\gamma'}\bigl(\Gamma(\alpha)\Gamma(\beta)\bigr)\tens\uprho_{\overline{\beta}}\bigl(\Gamma(\beta)^{-1}\bigr)\bigr)V(\alpha,\gamma'\tp\overline{\beta},i)\\
&=d_\alpha\Gamma(\beta)\bigl(\dim{\cH_{\overline{\beta}}\bigl(\Gamma(\beta)^{-1}\bigr)}\bigr)\uprho_\alpha\bigl(\Gamma(\alpha)\bigr)\\
&=d_\alpha\Gamma(\beta)\bigl(\dim{\cH_\beta\bigl(\Gamma(\beta)\bigr)}\bigr)\uprho_\alpha\bigl(\Gamma(\alpha)\bigr).
\end{split}
\end{equation}
Taking norm of both sides of \eqref{najpierw} and using Propositions \ref{rhoprop}, \ref{mprop}, \ref{gammaprop} we get
{\allowdisplaybreaks
\begin{align*}
d_\alpha\Gamma(\beta)&\dim\cH_\beta\bigl(\Gamma(\beta)\bigr)\\
&=\Bigl\|\sum_{\gamma'\in\Irr{\GG}}\!\!\!\!\!\sum_{i=1}^{m(\alpha,\gamma'\stp\overline{\beta})}\!\!\!d_{\gamma'}V(\alpha,\gamma'\tp\overline{\beta},i)^*\bigl(\uprho_{\gamma'}\bigl(\Gamma(\alpha)\Gamma(\beta)\bigr)\tens\uprho_{\overline{\beta}}\bigl(\Gamma(\beta)^{-1}\bigr)\bigr)V(\alpha,\gamma'\tp\overline{\beta},i)\Bigr\|\\
&\leq\!\!\sum_{\gamma'\in\Irr{\GG}}\!\!m(\alpha,\gamma'\tp\overline{\beta})d_{\gamma'}\bigl\|\uprho_{\gamma'}\bigl(\Gamma(\alpha)\Gamma(\beta)\bigr)\bigr\|\\
&=\!\!\sum_{\gamma'\in\Irr{\GG}}\!\!m(\gamma',\alpha\tp\beta)d_{\gamma'}\bigl\|\uprho_{\gamma'}\bigl(\Gamma(\alpha)\Gamma(\beta)\bigr)\bigr\|\\
&= \!\!\!\sum_{\substack{\gamma'\in\Irr{\GG}:\\\Gamma(\gamma')=\Gamma(\alpha)\Gamma(\beta)}}\!\!\!\!m(\gamma',\alpha\tp\beta)d_{\gamma'}\\
&=\!\!\!\sum_{\substack{\gamma'\in\Irr{\GG}:\\\Gamma(\gamma')=\Gamma(\alpha)\Gamma(\beta)}}\!\!\!\!m(\gamma',\alpha\tp\beta)\tfrac{d_{\gamma'}}{\dim{\cH_{\gamma'}(\Gamma(\gamma'))}}\dim{\cH_{\gamma'}(\Gamma(\gamma'))}\\
&\leq\tfrac{d_\gamma}{\dim{\cH_\gamma(\Gamma(\gamma))}}\!\sum_{\substack{\gamma'\in\Irr{\GG}:\\\Gamma(\gamma')=\Gamma(\alpha)\Gamma(\beta)}}\!\!\!\!m(\gamma',\alpha\tp\beta)\dim{\cH_{\gamma'}(\Gamma(\gamma'))}\\
&=\tfrac{d_\gamma}{\dim{\cH_\gamma(\Gamma(\gamma))}}\bigl(\dim{\cH_\alpha\bigl(\Gamma(\alpha)\bigr)}\bigr)\bigl(\dim{\cH_\beta\bigl(\Gamma(\beta)\bigr)}\bigr),
\end{align*}
}which yields \eqref{dimgammaineq}.
\end{proof}

Now we are able to prove the main theorem of the paper:

\begin{theorem}\label{mainThm}
Assume that $\bN_{\GG}<+\infty$. Then $\GG$ is of Kac type.
\end{theorem}

The remainder of this section (apart from Corollary \ref{onerep}) will be devoted to the proof of Theorem \ref{mainThm}. Case $\bN_\GG=1$ is trivial, hence assume that $\bN_\GG\geq{2}$. Assume by contradiction that $\GG$ is not of Kac type. Then there exists $\alpha\in\Irr{\GG}$ such that $\Gamma(\alpha)>1$.

We now proceed to choose a sequence $(\alpha_k)_{k\in\NN}$ of elements of $\Irr{\GG}$ such that $\alpha_1=\alpha$ (as above),
\newcounter{c}
\begin{enumerate}
\item \label{alphak1} $\alpha_{k+1}\cleq\alpha_{k}\tp\alpha_{k}$,
\item \label{alphak2} $\Gamma(\alpha_{k+1})=\Gamma(\alpha_{k})^{2}$,\setcounter{c}{\value{enumi}}
\end{enumerate}
and
\begin{enumerate}
\setcounter{enumi}{\value{c}}\item \label{alphak3}
$
\tfrac{d_{1}(\alpha_{k+1})}{\dim{\cH_{\alpha_{k+1}}(\Gamma(\alpha_{k+1}))}}
=\max\Bigl\{\tfrac{d_{1}(\gamma)}{\dim{\cH_\gamma(\Gamma(\gamma))}}\st\gamma\in\Irr{\GG},\:\gamma\cleq\alpha_{k}\tp\alpha_{k},\:\Gamma(\gamma)=\Gamma(\alpha_{k})^{2}\Bigr\}
$
\end{enumerate}
for all $k\in\NN$. Property \eqref{alphak2} implies
\[
\Gamma(\alpha_{k})=\Gamma(\alpha)^{2^{(k-1)}}
\]
for every $k\in\NN$.

We will continue to refine our sequence by choosing appropriate subsequences in order to finally arrive at a contradiction. We begin with the following Lemma:

\begin{lemma}\label{kaclemma1}
Let $(k_n)_{n\in\NN}$ be a strictly increasing sequence of natural numbers. Then
\[
1\leq\frac{d_1(\alpha_{k_{n+1}})\dim{\cH_{\alpha_{k_n}}\bigl(\Gamma(\alpha_{k_n})\bigr)}}{d_1(\alpha_{k_n})\dim{\cH_{\alpha_{k_{n+1}}}\bigl(\Gamma(\alpha_{k_{n+1}})\bigr)}}
\Gamma(\alpha)^{2^{(k_n-1)}-2^{(k_{n+1}-1)}}
\]
for every $n\in\NN$.
\end{lemma}

\begin{proof}
Fix $n\in\NN$. For each $k\in\{k_n,\dotsc,k_{n+1}-1\}$ representations $\alpha_k,\alpha_{k+1}$ satisfy assumptions of Proposition \ref{dimgammaineq}. Therefore
{\allowdisplaybreaks
\begin{align*}
1&\leq
\prod_{k=k_n}^{k_{n+1}-1}\!\frac{d_{1}(\alpha_{k+1})\dim{\cH_{\alpha_k}\bigl(\Gamma(\alpha_k)\bigr)}}{d_{1}(\alpha_k)\Gamma(\alpha_k)\dim{\cH_{\alpha_{k+1}}\bigl(\Gamma(\alpha_{k+1})\bigr)}}\\
&=\frac{d_{1}(\alpha_{k_{n+1}})\dim{\cH_{\alpha_{k_n}}\bigl(\Gamma(\alpha_{k_n})\bigr)}}{d_{1}(\alpha_{k_n})\dim{\cH_{\alpha_{k_{n+1}}}\bigl(\Gamma(\alpha_{k_{n+1}})\bigr)}}\prod_{k=k_n}^{k_{n+1}-1}\Gamma(\alpha)^{-2^{(k-1)}}\\
&=\frac{d_{1}(\alpha_{k_{n+1}})\dim{\cH_{\alpha_{k_n}}\bigl(\Gamma(\alpha_{k_n})\bigr)}}{d_{1}(\alpha_{k_n})\dim{\cH_{\alpha_{k_{n+1}}}\bigl(\Gamma(\alpha_{k_{n+1}})\bigr)}}\,
\Gamma(\alpha)^{-\!\!\!\!\sum\limits_{k=k_n}^{(k_{n+1}-1)}2^{(k-1)}}\\
&=\frac{d_{1}(\alpha_{k_{n+1}})\dim{\cH_{\alpha_{k_n}}\bigl(\Gamma(\alpha_{k_n})\bigr)}}{d_{1}(\alpha_{k_n})\dim{\cH_{\alpha_{k_{n+1}}}\bigl(\Gamma(\alpha_{k_{n+1}})\bigr)}}\,\Gamma(\alpha)^{2^{(k_n-1)}-2^{(k_{n+1}-1)}}.
\end{align*}
}\end{proof}

We know from Corollary \ref{symcor} that for each $k$ the operator $\uprho_{\alpha_k}$ can be expressed as 
\[
\uprho_{\alpha_k}=\operatorname{diag}\bigl(\Gamma(\alpha_k),\Gamma(\alpha_k)^{\theta^k_2},\dotsc,\Gamma(\alpha_k)^{\theta^k_{n_{\alpha_k}/2}},\Gamma(\alpha_k)^{-\theta^k_{n_{\alpha_k}/2}},\dotsc,\Gamma(\alpha_k)^{-\theta^k_2},\Gamma(\alpha_k)^{-1}\bigr)
\]
when $n_{\alpha_k}$ is even and 
\[
\uprho_{\alpha_k}=\operatorname{diag}\bigl(\Gamma(\alpha_k),\Gamma(\alpha_k)^{\theta^k_2},\dotsc,\Gamma(\alpha_k)^{\theta^k_{\lfloor{n_{\alpha_k}/2}\rfloor}},1,\Gamma(\alpha_k)^{-\theta^k_{\lfloor{n_{\alpha_k}/2}\rfloor}},\dotsc,\Gamma(\alpha_k)^{-\theta^k_2},\Gamma(\alpha_k)^{-1}\bigr)
\]
when $n_{\alpha_k}$ is odd for some non-negative numbers $\theta^{k}_{2},\dotsc,\theta^{k}_{\lfloor{n_{\alpha_k}/2}\rfloor}\in[0,1]$. For notational convenience we will also put $\theta^k_1=1$, so that
\[
\uprho_{\alpha_k}=\operatorname{diag}\bigl(\Gamma(\alpha_k)^{\theta^k_1},\dotsc,\Gamma(\alpha_k)^{\theta^k_{n_{\alpha_k}/2}},\Gamma(\alpha_k)^{-\theta^k_{n_{\alpha_k}/2}},\dotsc,\Gamma(\alpha_k)^{-\theta^k_1})\\
\]
or
\[
\uprho_{\alpha_k}=\operatorname{diag}\bigl(\Gamma(\alpha_k)^{\theta^k_1},\dotsc,\Gamma(\alpha_k)^{\theta^k_{\lfloor{n_{\alpha_k}/2}\rfloor}},1,\Gamma(\alpha_k)^{-\theta^k_{\lfloor{n_{\alpha_k}/2}\rfloor}},\dotsc,\Gamma(\alpha_k)^{-\theta^k_1}\bigr)
\]
depending on the parity of $n_\alpha$.

We will now show that there is a subsequence $(\alpha_{k_n})_{n\in\NN}$ of $(\alpha_k)_{k\in\NN}$ such that
\begin{enumerate}
\item\label{warunek1} $n_{\alpha_{k_n}}=N$ for each $n\in\NN$ and some $N\in\bigl\{2,\dotsc,\bN_\GG\bigr\}$,
\item\label{warunek3} for each $n\in\NN$ we have $k_{n+1}-k_n\geq{2}$,
\item\label{warunek4} for each $n\in\NN$ and $j\in\bigl\{1,\dotsc,\lfloor N/2\rfloor\bigr\}$ we have
\[
\Gamma(\alpha)^{2^{(k_{n+1}-1)}\theta^{k_{n+1}}_j}\leq\Gamma(\alpha)^{2^{(k_{n+1}-1)}-2^{(k_n-1)}}\Gamma(\alpha)^{2^{(k_n-1)}\theta^{k_n}_j}.
\] 
\end{enumerate}

It is easy to see that there is a a subsequence $(\alpha_{k^0_n})_{n\in\NN}$ of $(\alpha_k)_{k\in\NN}$ satisfying \eqref{warunek1} and \eqref{warunek3}. In order to see that we can refine it so that the resulting subsequence also satisfies \eqref{warunek4} we note first that it follows from the construction of the sequence $(\alpha_k)_{k\in\NN}$ that we have
\[
\operatorname{Sp}\bigl(\uprho_{\alpha_{k^0_n}}\bigr)\subset\underbrace{\operatorname{Sp}(\uprho_\alpha)\cdot\dotsc\cdot\operatorname{Sp}(\uprho_\alpha)}_{2^{(k^0_n-1)}}
\]
for each $n\in\NN$. Now for each $j\in\bigl\{1,\dots,\lfloor{N/2}\rfloor\bigr\}$ the number $\Gamma(\alpha)^{2^{k^0_n-1}\theta^{k^0_n}_j}$ belongs to the spectrum of $\uprho_{\alpha_{k^0_n}}$, so it can be written as a product of eigenvalues $\bigl\{\lambda_1,\dotsc,\lambda_{\#\operatorname{Sp}(\uprho_\alpha)}\bigr\}$ in appropriate powers:
\[
\Gamma(\alpha)^{2^{k^0_n-1}\theta^{k^0_n}_j}=\prod\limits_{m=1}^{\#\operatorname{Sp}(\uprho_{\alpha})}\lambda_m^{d(k^0_n,m,j)},
\]
\sloppy
where the non-negative integers ${\bigl\{d(k^0_n,m,j)\st}m\in\{1,\dotsc,\#\operatorname{Sp}(\uprho_\alpha)\},\:j\in\{1,\dots,\lfloor{N/2}\rfloor\}\bigr\}$ satisfy
\[
\sum\limits_{m=1}^{\#\operatorname{Sp}(\uprho_\alpha)}d(k^0_n,m,j)=2^{(k^0_n-1)},\qqquad{n}\in\NN.
\]

By choosing an appropriate subsequence $(k_n)_{n\in\NN}$ of $(k^0_n)_{n\in\NN}$ we can arrange that we have $d(k_{n+1},m,j)\geq{d(k_n,m,j)}$ for all $m$ and a fixed $j$. It remains to repeat this procedure for all $j$ refining the sequence each time. Having done so, let us keep the notation $(k_n)_{n\in\NN}$ for the resulting sequence of natural numbers. We have
\[
\begin{split}
\Gamma(\alpha)^{2^{(k_{n+1}-1)}\theta^{k_{n+1}}_j}
&=\prod_{m=1}^{\#\operatorname{Sp}(\uprho_\alpha)}\lambda_m^{d(k_{n+1},m,j)}\\
&=\biggl(\prod_{m=1}^{\#\operatorname{Sp}(\uprho_\alpha)}\lambda_m^{\bigl(d(k_{n+1},m,j)-d(k_n,m,j)\bigr)}\biggr)
\biggl(\prod_{m=1}^{\#\operatorname{Sp}(\uprho_\alpha)}\lambda_m^{d(k_n,m,j)}\biggr)\\
&\leq\Gamma(\alpha)^{\sum_{m=1}^{\#\operatorname{Sp}(\uprho_\alpha)}\bigl(d(k_{n+1},m,j)-d(k_n,m,j)\bigr)}\prod_{m=1}^{\#\operatorname{Sp}(\uprho_\alpha)}\lambda_m^{d(k_n,m,j)}\\
&=\Gamma(\alpha)^{2^{(k_{n+1}-1)}-2^{(k_n-1)}}\Gamma(\alpha)^{2^{(k_n-1)}\theta^{k_n}_j}
\end{split}
\]
(the inequality follows from the fact that $\Gamma(\alpha)=\max\{\lambda_1,\dotsc,\lambda_{\#\operatorname{Sp}(\uprho_\alpha)}\}$). In other words $(\alpha_{k_n})_{n\in\NN}$ satisfies conditions \eqref{warunek1}--\eqref{warunek4}.

Now using Lemma \ref{kaclemma1} and properties \eqref{warunek1}--\eqref{warunek4} of $(\alpha_{k_n})_{n\in\NN}$ we will arrive at a contradiction. Set $t=0$ if $N$ is even and $t=1$ otherwise. We have
\begin{equation}\label{nierownosc}
\begin{split}
1&\leq\frac{d_1(\alpha_{k_{n+1}})}{d_1(\alpha_{k_n})}\Gamma(\alpha)^{2^{(k_n-1)}-2^{(k_{n+1}-1)}}\\
&=\frac{\sum\limits_{j=1}^{\lfloor{N/2}\rfloor}\biggl(\Gamma(\alpha)^{2^{(k_{n+1}-1)}\theta^{k_{n+1}}_j}+\Gamma(\alpha)^{-2^{(k_{n+1}-1)}\theta^{k_{n+1}}_j}\biggr)+t}
{\sum\limits_{j=1}^{\lfloor{N/2}\rfloor}\biggl(\Gamma(\alpha)^{2^{(k_n-1)}\theta^{k_n}_j}+\Gamma(\alpha)^{-2^{(k_n-1)}\theta^{k_n}_j}\biggr)+t}\Gamma(\alpha)^{2^{(k_n-1)}-2^{(k_{n+1}-1)}}\\
&\leq
\frac{\sum\limits_{j=1}^{\lfloor{N/2}\rfloor}\biggl(\Gamma(\alpha)^{2^{(k_n-1)}\theta^{k_n}_j}+\Gamma(\alpha)^{2^{(k_n-1)}-2^{(k_{n+1}-1)}\bigl(\theta^{k_{n+1}}_j+1\bigr)}\biggr)+t\Gamma(\alpha)^{2^{(k_n-1)}-2^{(k_{n+1}-1)}}}
{\sum\limits_{j=1}^{\lfloor{N/2}\rfloor}\biggl(\Gamma(\alpha)^{2^{(k_n-1)}\theta^{k_n}_j}+\Gamma(\alpha)^{-2^{(k_n-1)}\theta^{k_n}_j}\biggr)+t}.
\end{split}
\end{equation}

Since $(k_n)_{n\in\NN}$ is an increasing sequence, clearly
\begin{equation}\label{tt}
t\Gamma(\alpha)^{2^{(k_n-1)}-2^{(k_{n+1}-1)}}\leq{t},\qqquad{n}\in\NN. 
\end{equation}
Moreover, for each $j\in\bigl\{1,\dotsc,\lfloor N/2\rfloor\bigr\}$ and $n\in\NN$ we have
\[
2^{(k_{n+1}-1)}\bigl(\theta^{k_{n+1}}_j+1\bigr)\geq2^{(k_{n+1}-1)}>2^{k_n}\geq2^{(k_n-1)}\theta^{k_n}_j+2^{(k_n-1)}
\]
and hence
\begin{equation}\label{kk}
2^{(k_n-1)}-2^{(k_{n+1}-1)}\bigl(\theta^{k_{n+1}}_j+1\bigr)<-2^{(k_n-1)}\theta^{k_n}_j.
\end{equation}
Comparing appropriate terms in the numerator and denominator of the right hand side of \eqref{nierownosc} and using \eqref{tt} and \eqref{kk} we find that
\[
\frac{\sum\limits_{j=1}^{\lfloor{N/2}\rfloor}\biggl(\Gamma(\alpha)^{2^{(k_n-1)}\theta^{k_n}_j}+\Gamma(\alpha)^{2^{(k_n-1)}-2^{(k_{n+1}-1)}\bigl(\theta^{k_{n+1}}_j+1\bigr)}\biggr)+t\Gamma(\alpha)^{2^{(k_n-1)}-2^{(k_{n+1}-1)}}}
{\sum\limits_{j=1}^{\lfloor{N/2}\rfloor}\biggl(\Gamma(\alpha)^{2^{(k_n-1)}\theta^{k_n}_j}+\Gamma(\alpha)^{-2^{(k_n-1)}\theta^{k_n}_j}\biggr)+t}<1
\]
which contradicts $\eqref{nierownosc}$ and therefore proves Theorem \ref{mainThm}.

At the end of this section we use Theorem \ref{mainThm} to derive a corollary concerning quantum groups which are not of Kac type.

\begin{corollary}\label{onerep}
Let $\GG$ be a compact quantum group and let $U\in\C(\GG)\tens\B(\cH_{U})$ be a finite dimensional unitary representation such that $\Gamma(U)>1$. Then 
\[
\sup{\bigl\{n_{\beta}\st}\beta\in\Irr{\GG},\:
\beta\cleq{U^{\stp{k_1}}\tp\dotsm\tp U^{\stp{k_n}}},\:n\in\ZZ_+,\:k_1,\dotsc,k_n\in\ZZ\bigr\}=+\infty,
\]
where we have used conventions $U^{-n}=\overline{U}^n$ for $n\in\NN$ and $U^0=\I$.
\end{corollary}

\begin{proof}
As any unitary representation decomposes into sum of irreducible ones, it is enough to prove this claim for $U=U^{\alpha}$, where $\alpha\in\Irr{\GG}$. Let $\HH$ be the image of $\GG$ in the representation $U$, i.e.~$\C(\HH)$ is the \cst-algebra generated by $\bigl\{U^\alpha_{i,j}{\st{i,j}\in\{1,\dotsc,n_\alpha\}\bigr\}}$ and $\Delta_\HH=\bigl.\Delta_\GG\bigr|_{\C(\HH)}$ (cf.~\cite[Remarks 3 and 1]{spheres}). It is easily seen that $\bigl.\bh_\GG\bigr|_{\C(\HH)}$ is a bi-invariant state on $\C(\HH)$ and consequently, by uniqueness of Haar measure, we have $\bh_{\HH}=\bigl.\bh_\GG\bigr|_{\C(\HH)}$.

For some $n\in\ZZ_+$, and $k_1,\dotsc,k_n\in\ZZ$ let $\beta\cleq\alpha^{\stp{k_1}}\tp\dotsm\tp\alpha^{\stp{k_n}}$ be a (class of) an irreducible representation of $\GG$. Then $U^{\beta}$ is also irreducible as a representation of $\HH$ since
\[
\bh_\HH({\chi_\beta}^*\chi_\beta)=\bh_\GG({\chi_\beta}^*\chi_\beta)=1
\]
(where $\chi_\beta$ is the character of $U^\beta$, cf.~\cite[Corollary 5.10]{pseudogr}). Since matrix elements of such representations span dense subspace in $\C(\HH)$ we have 
\[
\Irr{\HH}={\bigl\{\beta\in\Irr{\GG}\st}\beta\in\Irr{\GG},\:
\beta\cleq{U^{\stp{k_1}}\tp\dotsm\tp U^{\stp{k_n}}},\:n\in\ZZ_+,\:k_1,\dotsc,k_n\in\ZZ\bigr\}.
\]
Since $\Gamma(\alpha)>1$ (where $\alpha$ is considered as a class of representation of $\HH$) we must have
\[
\sup{\bigl\{n_{\beta}\st}\beta\in\Irr{\GG},\:\beta\cleq{U^{\stp{k_1}}\tp\dotsm\tp U^{\stp{k_n}}},\:n\in\ZZ_+,\:k_1,\dotsc,k_n\in\ZZ\bigr\}=+\infty
\]
due to Theorem \ref{mainThm}.
\end{proof}

\renewcommand{\thesection}{A}
\setcounter{proposition}{0}
\setcounter{equation}{0}

\section*{Appendix: Algebraic characterization of groups with representations of bounded degree}

In this section we note a characterization of the property of $\GG$ having irreducible representations of bounded degree in terms of the comultiplication on $\Pol(\GG)$. The reasoning is based on the fact that the algebra of $n\times{n}$ matrices (over a field of characteristic $0$) satisfies a polynomial identity of degree $2n$ and not lower (cf.~\cite{PIalg}).

\begin{proposition}\label{rn}
Let $\GG$ be a compact quantum group. Then $\bN_\GG<+\infty$ if and only if there exists $r\geq{2}$ such that
\begin{equation}\label{IK}
\sum_{\pi\in\mathsf{S}_r}\operatorname{sgn}(\pi)x_{\pi(1)}{\dotsm}x_{\pi(r)}=0,\qqquad{x_1,\dotsc,x_r}\in\operatorname{c}_{00}(\hh{\GG}).
\end{equation}
\end{proposition}

\begin{proof}
By the Amitsur-Levitzki theorem \cite[Section 4]{PIalg} for any $n\geq{2}$ we have
\[
\sum_{\pi\in\mathsf{S}_{2n}}\operatorname{sgn}(\pi)m_{\pi(1)}{\dotsm}m_{\pi(2n)}=0,\qqquad{m_1,\dotsc,m_{2n}}\in\mathsf{M}_n(\CC).
\]
Moreover $\mathsf{M}_n(\CC)$ does not have a proper polynomial identity of degree strictly smaller than $2n$ (\cite[Section 3, Lemma 2]{PIalg}). Since $\operatorname{c}_{00}(\hh{\GG})$ is the algebraic direct sum of matrix algebras of sizes equal to the dimensions of irreducible representations of $\GG$, we see that $\bN_\GG$ is finite if and only if \eqref{IK} is satisfied for some $r$ (namely $r=2\bN_\GG$ or larger).
\end{proof}

Condition \eqref{IK} form Proposition \ref{rn} can be rewritten in the following way: for $\pi\in\mathsf{S}_r$ let $\widetilde{\pi}$ be the operator on $\operatorname{c}_{00}(\hh{\GG})^{\tens{r}}$ permuting the tensor factors and let $\boldsymbol{\mu}$ be the multiplication map $\operatorname{c}_{00}(\hh{\GG})\tens\operatorname{c}_{00}(\hh{\GG})\to\operatorname{c}_{00}(\hh{\GG})$. Further let $\bigl(\boldsymbol{\mu}^{(k)}\bigr)_{k\in\NN}$ be the obvious extensions of multiplication to higher tensor powers of $\operatorname{c}_{00}(\hh{\GG})$:
\[
\boldsymbol{\mu}^{(k)}:\operatorname{c}_{00}(\hh{\GG})^{\tens(k+1)}\longrightarrow\operatorname{c}_{00}(\hh{\GG}),\qqquad{k}\in\NN.
\]
Then \eqref{IK} means simply
\[
\sum_{\pi\in\mathsf{S}_r}\operatorname{sgn}(\pi)\cdot\boldsymbol{\mu}^{(r-1)}\comp\widetilde{\pi}=0.
\]

Now recall that $\Pol(\GG)$ is the (multiplier) Hopf algebra dual to $\operatorname{c}_{00}(\hh{\GG})$ (\cite{vanDaele}). In particular, for each $k\geq{2}$ the map $\boldsymbol{\mu}^{(k)}$ is dual to 
\[
\Delta_\GG^{(k)}:\operatorname{c}_{00}(\hh{\GG})\longrightarrow\operatorname{M}\bigl(\operatorname{c}_{00}(\hh{\GG})^{\tens(k+1)}\bigr),
\]
where $\operatorname{M}(\,\cdot\,)$ denotes the multiplier functor (\cite{mha,vanDaele}). Thus the condition of having irreducible representations of bounded degree can be expressed in terms of the coalgebra structure of $\Pol(\GG)$:

\begin{corollary}
Let $\GG$ be a compact quantum group. Then $\bN_\GG<+\infty$ if and only if there exists $r\geq{2}$ such that
\begin{equation}\label{Pr}
\sum_{\pi\in\mathsf{S}_r}\operatorname{sgn}(\pi)\cdot\widetilde{\pi}\comp\Delta_\GG^{(r-1)}=0.
\end{equation}
\end{corollary}

\begin{remark}
Let us note that it can be shown that for a \emph{classical} group $\GG=G$ condition \eqref{Pr} is equivalent to condition $P_r$ considered by Kaplansky (\cite[Section 3]{kaplansky}), which for connected $G$ is further equivalent to commutativity of the group.
\end{remark}

\subsection*{Acknowledgement}
Research presented in this paper was partially supported by the NCN (National Science Centre, Poland) grant no.~2015/17/B/ST1/00085.

\bibliography{low14}{}
\bibliographystyle{plain}


\end{document}